\documentclass[a4paper, 11pt, underruledhead]{paper}
\usepackage{math}
\usepackage[mathscr]{euscript}
\usepackage{enumerate}
\newif\ifnoweb
\nowebtrue

\begin{document}



\let\goth\mathfrak

\def\myend{{}\hfill{\small$\bigcirc$}}

\newenvironment{ctext}{%
  \par
  \smallskip
  \centering
}{%
 \par
 \smallskip
 \csname @endpetrue\endcsname
}

\def\id{\mathrm{id}}

\newcommand{\msub}{\mbox{\large$\goth y$}}    
\def\suport{{\mathrm{supp}}}
\newcommand{\nat}{{N}}   
\def\bagnom#1#2{\genfrac{[}{]}{0pt}{}{#1}{#2}}

\def\kros(#1,#2){c_{\{ #1,#2 \}}}
\def\skros(#1,#2){{\{ #1,#2 \}}}
\def\LineOn(#1,#2){\overline{{#1},{#2}\rule{0em}{1,5ex}}}
\def\inc{\mathrel{\strut\rule{3pt}{0pt}\rule{1pt}{9pt}\rule{3pt}{0pt}\strut}}
\def\lines{{\cal L}}
\def\tran{{\mathscr T}}
\def\collin{\sim}
\def\wspolin{{\bf L}}
\def\ncollin{\not\collin}

\def\VerSpace(#1,#2){{\bf V}_{{#2}}({#1})}
\def\GrasSpace(#1,#2){{\bf G}_{{#2}}({#1})}
\def\VeblSpSymb{{\bf P}\mkern-10mu{\bf B}}
\def\VeblSpace(#1){\VeblSpSymb({#1})}
\def\xwpis#1#2{{{\triangleright}\mkern-4mu{{\strut}_{{#1}}^{{#2}}}}}
\def\MultVeblSymb{{\sf M}\mkern-10mu{\sf V}}        
\def\MultVeblSymq{{\sf T}\mkern-6mu{\sf P}}        
\def\xwlep(#1,#2,#3,#4){{\MultVeblSymb^{#1}{\xwpis{#2}{#3}}{#4}}}
\def\xwlepp(#1,#2,#3,#4,#5){{\MultVeblSymb_{#1}^{#2}{\xwpis{#3}{#4}}{#5}}}
\def\xwlepq(#1,#2,#3,#4,#5){{\MultVeblSymq_{#1}^{#2}{\xwpis{#3}{#4}}{#5}}}
\def\konftyp(#1,#2,#3,#4){\left( {#1}_{#2}\, {#3}_{#4} \right)}
\def\binokonf(#1){\konftyp({\binom{{#1}}{2}},{{#1}-2},\binom{{#1}}{3},3)}
\let\binoconf\binokonf


\newcount\liczbaa
\newcount\liczbab

\def\binkonfo(#1,#2){\liczbaa=#2 \liczbab=#2 \advance\liczbab by -2
\def\doa{\ifnum\liczbaa = 0\relax \else
\ifnum\liczbaa < 0 \the\liczbaa \else +\the\liczbaa\fi\fi}
\def\dob{\ifnum\liczbab = 0\relax \else
\ifnum\liczbab < 0 \the\liczbab \else +\the\liczbab\fi\fi}
\konftyp(\binom{#1\doa}{2},#1\dob,\binom{#1\doa}{3},3) }

\newcount\liczbac
\def\binkonf(#1,#2){\liczbac=#2 
\def\docc{\ifnum\liczbac = 0 \relax \else\ifnum\liczbac < 0\the\liczbac 
\else+\the\liczbac\fi\fi}
B_{#1{\docc}}}


\def\Ver{{\sf\bfseries VER}}
\def\Des{{\sf\bfseries DES}}
\def\Desv{{\sf\bfseries DES'}}
\def\Desr{{\sf\bfseries DES''}}
\def\MTC{{\sf STP}}
\def\MVC{{\sf MVC}}

\def\PSTS{{\sf PSTS}}
\def\BSTS{{\sf BSTS}}


\title[Binomial \PSTS's containing complete graphs]{%
Binomial partial Steiner triple systems containing complete graphs}

\author{
Ma{\l}gorzata Pra{\.z}mowska%
\footnote{{\em Corresponding Author}, e-mail: {\ttfamily malgpraz@math.uwb.edu.pl},
telephone: +48 606 74 64 56}, 
Krzysztof Pra{\.z}mowski}

\maketitle

\begin{abstract}
  We propose a new approach to studies on partial Steiner triple systems consisting
  in determining complete graphs contained in them.
  We establish the structure which complete graphs yield in a minimal \PSTS\ that
  contains them. As a by-product we introduce the notion of a binomial \PSTS\
  as a configuration with parameters of a minimal \PSTS\ with a complete subgraph.
  A representation of binomial \PSTS\ with at least a given number of its maximal
  complete subgraphs is given in terms of systems of perspectives.
  \ifnoweb
  Finally, we prove that for each admissible integer there is a binomial \PSTS\ with
  this number of maximal complete subgraphs. 
  \fi
  \begin{flushleft}
  \end{flushleft}
\end{abstract}

\section*{Introduction}

In the paper we investigate the structure which (may) yield complete graphs 
contained in a (partial) Steiner triple system (in short: in a \PSTS).
Our problem is, in fact, a particular instance of a general question,
investigated in the literature, which {\sf STS}'s (more generally: which
{\sf PSTS}'s) contain/do not contain a configuration of a prescribed
type.
Here belongs, e.g. the problem to determine Pasch-free configurations
in some definite classes of ({\sf P}){\sf STS}'s  (\cite{antipasch1}), 
or to characterize {\sf STS}'s
which do not contain the mitre configuration (e.g.: \cite{antimitre1}
\cite{antimitre2})). 
It resembles also (and, in a sense, generalizes) the
problem to 
determine triangle-free (e.g. \cite{trianglefree2}, \cite{trianglefree3}), 
quadrangle-free (e.g. \cite{trianglefree1}), 
and other $\cal K$-free (e.g. \cite{6sparse}, \cite{kwardfree})
configurations, where $\cal K$ is a fixed class of configurations (of a
special importance).
In our case configurations in question are 
complete graphs and so-called binomial configurations.

In essence, in what follows, speaking about a (complete) graph $G$ contained in
a configuration $\goth M$ we always assume that $G$ is {\em freely} contained 
in $\goth M$ i.e. it is not merely a subgraph of 
the collinearity (point-adjacency) graph of $\goth M$ 
but also distinct edges of $G$ lie on distinct lines (so called {\em sides})
of $\goth M$, and sides do not intersect outside $G$.
Note that following this terminology a complete quadrangle on a projective plane $\goth P$
{\em is not} a $K_4$-graph {\em contained} in $\goth P$.

\medskip
There are two main problems that this theory starts with.
Firstly: 
{\em what are the minimal parameters of a \PSTS\ necessary to contain a $K_n$-graph}
and 
{\em does there exist  a \PSTS\ with these parameters which contains/does not contain
a $K_n$-graph}.
It turns out that a minimal (with respect to its size) \PSTS\ that contains $K_n$ is a
so called {\em binomial configuration} i.e. a $\konftyp(v,r,b,3)$-configuration
such that
$v =\binom{m}{2}$, $r = m-2$, and $b = \binom{m}{3}$ for a positive integer $m$ ($=n+1$)
(Prop. \ref{minSTS2pelny}).
We say, in short, that 
this \PSTS\ is {\em a minimal configuration} which contains $K_n$.

Several classes of binomial configurations were introduced and studied in the literature
(see \cite[combinatorial Grassmannians]{perspect}, 
\cite[combinatorial Veronesians]{combver},
\cite[multiveblen configurations]{pascvebl}, \cite{mveb2proj}, \cite{skewgras}).
The above characterization of the parameters of a binomial configuration gives 
a good motivation and justification
for investigating this class from `a general perspective'.
In most of the binomial configurations already defined in the literature a suitable
complete graph can be found.
As the best known example of a binomial configuration
we can quote generalized Desargues configuration
(see \cite{klimczak}, \cite{doliwa1}, \cite{doliwa2}).
Then the answer to the `dual' question 
{\em what is the maximal size of a complete graph that a binomial
$\binkonfo(n,0)$-configuration $\goth M$ may contain}
easily follows: this size is $n-1$.
In this case we say that $K_n$ is a maximal (complete) subgraph of $\goth M$.
However, there are $\binkonfo(n,0)$-configurations that do not contain any $K_{n-1}$-graph.
The simplest example can be found for $n=5$: it is known that there are 
$10_3$-configurations without $K_4$.
(cf. \cite{betten}, \cite{klik:VC}).

\medskip
The second problem is read as follows: 
provided a $\binkonfo(n,1)$-configuration contains a $K_n$,
{\em what is the possible number of $K_n$-graphs contained in $\goth M$ and what is the
structure they yield}.
It turns  out that the maximal number of such $K_n$-subgraphs is $n+1$
(Prop. \ref{prop:maxK}).
Binomial configurations with the maximal number of $K_n$-subgraphs are exactly
the generalized Desargues configurations.
This fact points out once more a special position of the class of generalized 
Desargues configurations within the class of binomial configurations.

As we already mentioned,
the problem whether a given binomial $\binkonfo(n,1)$-configuration $\goth M$
contains a complete $K_n$-graph is slightly similar to the problem if a \PSTS\ contains 
a Pasch configuration. Indeed, $\goth M$ contains such a graph iff it contains
a binomial $\binkonfo(n,0)$-configuration 
(Prop. \ref{pelny2horyzont}, Cor.\ref{cor:pelny2horyzont}). 
So, (binomial) configurations 
(of size $\binom{n+1}{2}$)
without $K_n$-graphs are exactly the configurations whose no subspace is 
a binomial configuration of one-step-smaller size $\binom{n}{2}$.
\par
A binomial $\binkonfo(n,1)$-configuration with two $K_n$-subgraphs can be considered
as an abstract scheme of a perspective of two $n$-simplices
(Prop. \ref{prop:cross-compl3}).
Generally, the structure of the intersection points 
of $K_n$-subgraphs of a $\binkonfo(n,1)$-configuration
is isomorphic to a generalized Desargues configuration (Prop. \ref{prop:maxK}).

In the paper we do not intend to give a {\em detailed} analysis of the internal structure
of binomial configurations which 
contain a {\em prescribed} number of their maximal $K_n$-subgraphs.
We give only a technique (a procedure) to construct a binomial configuration which contains
at least the given number of $K_n$-subgraphs: Thm. \ref{thm:SPS}
And we prove that for each admissible integer $m$ there does exist a binomial configuration
which freely contains $m$ $K_n$-subgraphs.
Some remarks are also
made which show that, surprisingly, binomial configurations with `many' 
maximal complete subgraphs 
(the maximal admissible number, the maximal number reduced by 2, 
and the maximal number reduced by 3)
are the configurations of some well known classes.


\section{Definitions}


Recall that the term {\em combinatorial configuration} (or simply a {\em configuration}) is 
usually (cf. e.g. \cite{grop1}, \cite{steiniz}) 
used for a finite incidence point-line structure provided 
that two different points are incident with at most one line and the line size and 
point rank are constant.
Speaking more precisely, a {\em  $\konftyp(v,r,b,\varkappa)$-configuration} 
is a combinatorial configuration with 
$v$ points and $b$ lines such that
there are $r$ lines through each point, and there are $\varkappa$ points on each line.
A partial Steiner triple system (a \PSTS, in short) is a configuration 
whose each line contains exactly three points.

In what follows we shall pay special attention to so called 
{\em binomial configurations} (more precisely: binomial partial Steiner triple systems,
\BSTS, in short) i.e. to $\binkonfo(n,0)$-configurations with arbitrary integer $n\geq 2$.
The class of the configurations with these parameters will be denoted by $\binkonf(n,0)$.


Let $k$ be a positive integer and $X$ a set; we write $\sub_k(X)$ for the 
set of all $k$-subsets of $X$.
The incidence structure
\begin{ctext}
  $\GrasSpace(X,k) := \struct{\sub_k(X),\sub_{k+1}(X),\subset}$
\end{ctext}
will be called {\em a combinatorial Grassmannian} (cf. \cite{perspect}).
If $|X| = n$ then $\GrasSpace(X,2)$ is a $\binkonf(n,0)$-configuration, so it is 
a \BSTS. As the structure $\GrasSpace(X,2)$ is (up to an isomorphism) uniquely determined
by the cardinality $|X|$ in what follows we frequently write $\GrasSpace(|X|,2)$ instead
of $\GrasSpace(X,2)$. 
Recall that $\GrasSpace(4,2)$ is the Veblen (the Pasch) configuration and $\GrasSpace(5,2)$
is the Desargues Configuration (see e.g. \cite{combver}, \cite{perspect}).
Generally, every structure $\GrasSpace(n,2)$, $n\geq 5$ will be called
a {\em generalized Desargues configuration} 
(cf. \cite{skewgras}, \cite{doliwa1}, \cite{doliwa2}).

\medskip
Let $X$ be a non void set, $n = |X|$. A nondirected loopless graph
defined on $X$ (we say simply: a {\em graph})  is a structure of the form 
$\struct{X,{\cal P}}$ with ${\cal P}\subset \sub_2(X)$.
We write $K_X = \struct{X,\sub_2(X)}$ 
for the complete graph with the vertices $X$;
the term $K_n$ is used for an arbitrary graph $K_X$ where $|X| = n$.

We say that a configuration ${\goth N} = \struct{X,{\cal G}}$ 
{\em is contained in} a configuration ${\goth M} = \struct{S,\lines}$ if 
\begin{itemize}\itemsep-2pt\def\labelitemi{-}
\item
  $X \subset S$.
\item
  each line $L\in{\cal G}$ extends (uniquely) to some $\overline{L}\in\lines$, and
\item
  distinct lines of $\goth N$ extend to distinct lines of $\goth M$.
\end{itemize}
$\goth N$ is {\em l-closed} (is a {\em  subspace} of $\goth M$) if each line of 
$\goth M$ that crosses $X$ in at least two points is an extension of a line in $\cal G$.
In that case we also say that $\goth N$ is {\em regularly} contained in $\goth M$.
Frequently, with a slight abuse of language, we refer to lines of the form $\overline{L}$
($L\in{\cal G}$) as to {\em sides} of $\goth N$.
\par\noindent
${\goth N}$ is {\em p-closed} if its sides do no intersect outside $\goth N$ i.e. 
when the following holds:
\begin{ctext}
  $L_1,L_2\in{\cal G},\; L_1\neq L_2,\; p\in\overline{L_1}\cap \overline{L_2}
  \implies p \in X$
\end{ctext}
A graph contained and p-closed in $\goth M$ is said to be {\em freely contained in} $\goth M$.
In what follows the phrases
\begin{itemize}\def\labelitemi{}\itemsep-2pt
\item {\em $X$ is a $K_n$-graph (freely) contained in $\goth M$} and
\item {\em $K_X$ is (freely) contained in $\goth M$}
\end{itemize}
will be used equivalently, together with their (admissible) stylistic variants.


\section{Complete subgraphs freely contained in \PSTS's}

Clearly, there are configurations with freely contained subgraphs.
%
\begin{prop}\label{pelny2minSTS}[{\normalfont cf. \cite{klimczak} or \cite{pascvebl}}]
  Let $n\geq 2$ be an integer. 
  $\GrasSpace(n+1,2)$ is a $\binkonf(n,+1)$-configuration which freely contains $K_n$.
\end{prop}
\begin{prop}\label{minSTS2pelny}
  Let $n\geq 2$ be an integer.
  A smallest \PSTS\ that freely contains the complete graph $K_n$ is a 
  $\binkonf(n,+1)$-configuration.
  Consequently, it is a \BSTS.
\end{prop}
\begin{proof}
  Let ${\goth M} = \struct{S,\lines}$ freely contain $K_X = \struct{X,\sub_2(X)}$,
  $|X| = n$. Then $X \subset S$ and for each edge $e\in\sub_2(X)$ there is a third point
  $\infty_e\in\overline{e}\setminus e$ of $\goth M$ on $\overline{e}$.
  Therefore,
  $|S| \geq |X| + |\sub_2(X)| = \binom{n+1}{2}$.
  Write $X^\infty = \{ e_\infty\colon e\in\sub_2(X) \}$. The rank of a point $x\in X$
  in $\goth M$ is at least $n-1$, so the rank of a point $q \in X^\infty$ is at 
  least $n-1$ as well. 
  On the other hand, there passes exactly one side of 
  $K_X$ through a point in $X^\infty$.
  Let $b_0$ be the number of lines through a point in $X^\infty$
  distinct from the sides of $K_X$. This number is minimal when the lines in question
  contain entirely points in $X^\infty$ and then through $q\in X^\infty$ there pass: 
  one  side of $K_X$ and lines contained in $X^\infty$.
  Consequently, 
  $3 b_0 \geq |\sub_2(X)|(n-2)$. This yields $b_0 \geq \binom{n}{3}$.
  Finally, 
  $|\lines| \geq |\sub_2(X)| + b_0 = \binom{n+1}{3}$.
\end{proof}
\begin{prop}\label{pelny2horyzont}
  Let ${\goth M} = \struct{S,\lines}$ be a minimal \PSTS\ which freely contains a complete
  graph $K_X = \struct{X,{\sub_2(X)}}$ and $|X| = n$. 
  Then {\em the complement of $K_X$} i.e. the structure
  $\struct{S\setminus X,\lines\setminus\{ \overline{e}\colon e\in{\sub_2(X)} \}}$
  is a
  $\binkonf(n,0)$-configuration and a subspace of $\goth M$.
  \par
  Conversely, let ${\goth N} = \struct{Z,{\cal G}}$ be a
  $\binkonf(n,0)$-configuration
  regularly contained in $\goth M$.
  Then $S\setminus Z$ yields in $\goth M$ a complete $K_n$-graph
  freely contained in $\goth M$, 
  whose complement is $\goth N$.
\end{prop}
\begin{proof}
  By \ref{minSTS2pelny}, $\goth M$ is a $\binkonf(n,+1)$-configuration.
  The first statement was proved, in fact, in the proof of \ref{minSTS2pelny}.
  \par
  Let $\goth N$ be a suitable subconfiguration of $\goth M$. Set $X = S\setminus Z$.
  Then $|X| = n$. Through each point $z\in Z$ there passes exactly one line 
  $L_z \in \lines\setminus{\cal G}$. If $z_1\neq z_2$ then $L_{z_1}\neq L_{z_2}$,
  as otherwise $L_{z_1}\in{\cal G}$.
  We have 
  $|\lines\setminus{\cal G}| = \binom{n+1}{3} - \binom{n}{3}
  = \binom{n}{2} = |\{L_z \colon z\in Z \}|$ 
  and therefore 
  $\lines\setminus{\cal G} = \{ L_z\colon z \in Z \}$.
  Each line in $\lines\setminus{\cal G}$ contains exactly two elements of $X$;
  comparing the parameters we get that 
  $\struct{X,\lines\setminus{\cal G}}$ is a complete graph.
\end{proof}

As we learn from the proofs of \ref{minSTS2pelny} and \ref{pelny2horyzont}
each minimal partial {\sf STS} that freely contains a complete graph $K_n$ is associated 
with a labelling (the map $e\mapsto\infty_e$) of the points of a 
$\binkonf(n,0)$-configuration $\goth H$ by the elements of $\sub_2(X)$, $|X| = n$.
\par
Indeed, let ${\goth H} = \struct{Z,{\cal G}}$, $\mu\colon \sub_2(X)\longrightarrow Z$
be a bijection, and $n = |X|$. Then the configuration
$$
  K_X +_\mu {\goth H} := \struct{{X \cup Z},%
  {{\cal G} \cup \left\{ \{ a,b,\mu(\{ a,b \}) \} \colon \{a,b\}\in\sub_2(X) \right\}}}
$$
%
is a $\binkonf(n,+1)$-configuration freely containing $K_X$.

In what follows any bijection $\mu$ of the points of a configuration onto an arbitrary
set will be frequently named a {\em labelling}. 
\footnote{%
In \cite{klimczak} the term `improper points' was used instead of `labelling'. 
Now we think that this older term should not be used, as it suggests, incorrectly,
some connections with a `parallelism', something to do with `directions'.}

The above apparatus was fruitfully applied in \cite{klik:VC} to the case $n=4$
(labelling of the Veblen configuration) to classify $10_3$-configurations which freely 
contain $K_4$. 
Clearly, this method can be applied to arbitrary $n$,
though even in the next step: $n=5$ a classification ``by hand" of all the 
labellings of arbitrary $\konftyp(\binom{5}{2},3,\binom{5}{3},3)$-configuration 
by the elements of 
$\sub_2(X)$ with $|X|=5$ seems seriously much more complex, if executable.
In what follows we shall propose a way that may be applied to arbitrary $n$
and which seems (at least a bit) less involved.

\begin{cor}\label{cor:pelny2horyzont}
  Let $\goth M$ be a $\binkonf(n,+1)$-configuration.
  $\goth M$ freely contains a complete graph $K_n$ iff $\goth M$
  regularly contains a $\binkonf(n,0)$-subconfiguration.
\end{cor}

\ifnoweb
\subsection{Intersection properties}
\fi

\begin{prop}\label{prop:cross-compl2}
  Any two distinct complete $K_n$-graphs freely contained in a \linebreak 
  $\binkonf(n,+1)$-configuration
  share exactly one vertex.
\end{prop}
\begin{proof}
  Let $G_1 = \struct{X_1,\sub_2(X_1)}$, $G_2 = \struct{X_2,\sub_2(X_2)}$ be two $K_n$
  graphs freely contained in a $\binkonf(n,+1)$-configuration ${\goth M}$.

  First, we note that $X_1\cap X_2\neq\emptyset$.
  Indeed, suppose that $X_1\cap X_2 = \emptyset$. 
  Then $X_2$ is freely contained in the complement of $X_1$.
  From \ref{pelny2horyzont}, this complement is a $\binkonf(n,0)$-configuration,
  which contradicts \ref{minSTS2pelny}.

  Let $a\in X_1\cap X_2$. Since 
  the degree of $a$ in $K_{X_1}$, in $\goth M$, and in $K_{X_2}$ is $n-1$, 
  $G_1$ and $G_2$ have common sides through $a$.
  Assume that $b \in X_1\cap X_2$ for some $b \neq a$; as previously the sides of $G_1$
  and of $G_2$ through $b$ coincide. So, consider arbitrary $x\in X_1\setminus\{ a,b\}$.
  $G_1$ and $G_2$ both contain the sides $\overline{\{ a,x \}}$ and $\overline{\{ x,b \}}$,
  so $x \in X_2$. Finally, we arrive to $X_1 = X_2$.
\end{proof}
For a geometer the situation considered in \ref{prop:cross-compl2} has clear geometrical
meaning: if $a$ is the common vertex of two complete $K_n$ graphs 
$\struct{X_1,{\cal E}_1}$, $\struct{X_2,{\cal E}_2}$ 
freely contained in a $\binkonf(n,+1)$-configuration $\goth M$ then $a$ is the 
{\em perspective center} of 
two $K_{n-1}$-simplices 
$A_1 = X_1\setminus\{a\}$ and $A_2 = X_2\setminus\{a\}$.
This means: there is a bijective correspondence $\sigma_a$ between the vertices of
the simplices such that for every vertex $x$ of the first simplex the corresponding
vertex $\sigma_a(x)$ of the latter simplex lies on the line $\overline{a,x}$
through $a$ and $x$.
As we shall see, in this case also an analogue of a {\em perspective axis} can be found.
That is, there is a subspace $Z$ of $\goth M$ and a bijective correspondence $\zeta$
between the sides of the simplices $A_1$ and $A_2$
such that for each side $L$ of the first simplex the corresponding side
$\zeta(L)$ of the latter simplex crosses $L$ in a point on $Z$.
\begin{prop}\label{prop:cross-compl3}
  Let $G_i= \struct{X_i,\sub_2(X_i)}$, $i=1,2$ be two complete $K_n$-graphs
  freely contained in a $\binkonf(n,+1)$-configuration ${\goth M}=\struct{S,\lines}$,
  let $p\in X_1\cap X_2$, and ${\goth N}_i$ with the pointset $Z_i = S\setminus X_{3-i}$
  be the complement of $G_{3-i}$ in $\goth M$ for $i=1,2$ (cf. \ref{pelny2horyzont}).
  \begin{sentences}\itemsep-2pt
  \item\label{cros-compl3:1}
    ${\goth N}_i$ freely contains a complete $K_{n-1}$-graph, for each $i=1,2$.
  \item\label{cros-compl3:2}
    Each side of $G_i$ missing $p$ crosses exactly one side of $G_{3-i}$ and
    the latter misses $p$ as well.
    The intersection points of the corresponding sides form the set $Z_1\cap Z_2$.
  \end{sentences}
\end{prop}
\begin{proof}
  It seen that 
  $X_{i}\setminus\{ p \}$ is a $K_{n-1}$-graph contained in $Z_{i}$ 
  and, clearly, it is p-closed. This proves \eqref{cros-compl3:1}.
  \par
  Let $p\notin e\in\sub_2(X_i)$, 
  then $\overline{e}\setminus e$ is a point $u_e$ in $S$.
  Since $u_e\notin X_i$, we have $u_e\in Z_{3-i}$.
  Moreover, $u_e\notin X_{3-i}$, since 
  $X_{3-i} \subset \bigcup\{ \overline{p,x}\colon x\in X_i \}$,
  so $u_e\in Z_i$.
  Therefore, there is a side $e'$
  of $G_{3-i}$ (a line of ${\goth N}_{3-i}$)
  which passes through $u_e$. 
  Statement \eqref{cros-compl3:2} is now evident.
\end{proof}
Even in the smallest possible case ($n = 4$) we have a 
$\binkonf(n,1)$-configuration (the fez configuration, cf. \cite{klik:VC}) 
which contains a pair of
perspective triangles with the perspective center $p$ such that the correspondence of 
the form $\overline{ \{ a,b\} } \longmapsto \overline{ \{ \sigma_p(a),\sigma_p(b) \} }$
does not yield any perspective axis, but the triangles in question do have a perspective axis.

\ifnoweb\relax\fi 

\begin{prop}\label{prop:cross-compl-line}
  Let $\struct{X_i,\sub_2(X_i)}$, $i=1,2,3$ be three distinct $K_n$ graphs 
  freely contained in a $\binkonf(n,+1)$-configuration $\goth M$. Let 
  $c_k\in X_i\cap X_j$ for all $\{k,i,j\} = \{1,2,3\}$.
  Then $\{c_1,c_2,c_3\}$ is a line of $\goth M$.
\end{prop}
\begin{proof}
  We have $c_3 \in X_1 \cap X_2$, so $X_2\setminus \{c_3\}$ consists of all the 
  `third points' on sides of $X_1$ through $c_3$ i.e. 
  \begin{ctext}
    $X_2\setminus \{c_3\} = \{ \overline{c_3,x}\setminus\{ c_3,x \}\colon 
    x\in X_1\setminus \{ c_3\} \}$.
  \end{ctext}
  From the assumption, $c_1\in X_2$ and thus $\{ c_1,c_3,u \}$ is a line of $\goth M$
  for some $u \in X_1\setminus\{ c_3 \}$.
  With analogous reasoning we have
  \begin{ctext}
    $X_3\setminus \{c_1\} = \{ \overline{c_1,x}\setminus\{ c_1,x \}\colon 
    x\in X_2\setminus \{ c_1\} \}$,
  \end{ctext}
  so $u\in X_3$. Finally, with \ref{prop:cross-compl2} we have $u = c_2$: the claim.
\end{proof}
\begin{cor}\label{cor:cross-compl-sides}
  Let ${\goth M} = \struct{S,\lines}$ be a $\binkonf(n,+1)$-configuration.
  Let \linebreak
  $X_1,X_2,X_3\in\sub_n(S)$ be pairwise distinct, ${\cal E}_i = \sub_2(X_i)$ and 
  $G_i = \struct{X_i,{\cal E}_i}$ for $i=1,2,3$.
  \begin{sentences}\itemsep-2pt
  \item
    Assume that $G_1$ and $G_2$ are freely contained in $\goth M$.
    Then $G_1,G_2$ have exactly $n-1$ common sides i.e.
    \begin{ctext}
      $|\left\{ \overline{e}\colon e\in{\cal E}_1 \right\} \cap 
      \left\{ \overline{e}\colon e\in{\cal E}_2 \right\}| = n-1$.
    \end{ctext}
  \item
    Assume that $G_1,G_2,G_3$ are freely contained in $\goth M$.
    Then there is exactly one side common to $G_1$, $G_2$, and $G_3$.
  \end{sentences}
\end{cor}
%

\subsection{The structure of complete subgraphs}

Next, we pass to an analysis of possible `many' subgraphs freely contained in 
a binomial configuration.
\begin{prop}\label{prop:maxK}
  Let ${G}_i = \struct{X_i,\sub_2(X_i)}$, $i=1,\ldots,m$ be a family of 
  $m$ distinct $K_n$-graphs
  freely contained in a $\binkonf(n,+1)$-configuration ${\goth M} = \struct{S,\lines}$.
  \begin{sentences}\itemsep-2pt
  \item\label{prop:maxK:1}
    Set $I = \{ 1,\ldots,m \}$.
    The map $q\colon \sub_2(I)\longrightarrow S$ determined (cf. \ref{prop:cross-compl2}) 
    by the condition
    \begin{ctext}
      $q^{i,j} = q(\{i,j\}) \in X_i \cap X_j$ for each $\{i,j \}\in\sub_2(I)$
    \end{ctext}
    embeds $\GrasSpace(I,2)$ into $\goth M$.
  \item\label{prop:maxK:2}
    Consequently, $m \leq n+1$.
  \item\label{prop:maxK:3}
    If $m=n$ then $\goth M$ freely contains one more, $(n+1)$-st $K_n$-graph. 
  \end{sentences}
\end{prop}
\begin{proof}
  Ad \eqref{prop:maxK:1}:\quad 
  From \ref{prop:cross-compl-line}, $X_i\cap X_j \cap X_k = \emptyset$ 
  for distinct $i,j,k$ in $I$ and thus the map $q$ is an injection.
  Moreover, \ref{prop:cross-compl-line} also yields that $q$ maps each line of 
  $\GrasSpace(I,2)$ onto a line of $\goth M$.
\par\noindent
  \eqref{prop:maxK:2} is immediate now.
\par\noindent
  Ad \eqref{prop:maxK:3}:\quad
  For each $i\in I$ there is exactly one point $d_i\in X_i\setminus\bigcup_{j\neq i}X_j$.
  Write $X_0 = \{ d_i\colon i\in I\}$, clearly, $|X_0| = n$.
  Let $i,j\in I$ be distinct; from definition, $q^{i,j}\neq d_i,d_j$.
  From  \ref{prop:cross-compl-line} 
  we get that
  for every $k\in I$, $k\neq i,j$ the side $\overline{q^{i,j},q^{i,k}}$ of $G_i$ crosses
  $X_j$ in the point $q^{j,k}$. So, the side $\overline{q^{i,j},d_i}$ of $G_i$ 
  crosses $X_j$ in a point distinct from all the $q^{j,k}$ i.e. in the point $d_j$.
  Thus $X_0$ is a complete graph with the sides
  $\{ d_i,d_j,q^{i,j} \}$, $\{ i,j \}\in\sub_2(I)$.
  It is seen that $X_0$ is freely contained in $\goth M$.
\end{proof}
\begin{cor}\label{cor:maxKgras}
  A $\binkonf(n,+1)$-configuration $\goth M$ freely contains $n+1$ $K_n$-graphs
  iff ${\goth M}\cong\GrasSpace(n+1,2)$.
\end{cor}
\begin{proof}
  In view of \ref{prop:maxK}\eqref{prop:maxK:1} it suffices to note that the sets 
  $S(i) = \{e\in\sub_2(I)\colon i\in e \}$ 
  are the maximal cliques of $\GrasSpace(I,2)$
  which are not lines of $\GrasSpace(I,2)$ (cf. \cite{perspect}). It is seen that
  each of them is freely contained in $\GrasSpace(I,2)$ for arbitrary set $I$
  with $|I|\geq 3$.
\end{proof}

The results obtained can be summarized in the following Proposition.
\begin{prop}\label{lem:maxK:struktura}
  Let ${G}_i = \struct{X_i,\sub_2(X_i)}$, $i=1,\ldots,m$ be a family of 
  $m$ distinct $K_n$-graphs
  freely contained in a $\binkonf(n,+1)$-configuration ${\goth M} = \struct{S,\lines}$.
  Set $I = \{ 1,\ldots,m \}$,
  $Z_i := X_i \setminus \bigcup_{k\in I\setminus\{ i \}} X_k$,
  $Z := S \setminus \bigcup_{i\in I}X_i$,
  ${\cal E}_i := \{ \overline{e}\colon e\in\sub_2(X_i) \}$,
  ${\cal G}_i := {\cal E}_i \setminus \bigcup_{k\in I\setminus \{ i \}} {\cal E}_k$,
  ${\cal G} := \lines\setminus \bigcup_{i\in I}{\cal E}_i$ for every $i\in I$,
  $q^{i,j}\in X_i,X_j$, $Q:=\{ q^{i,j}\colon \{ i,j\}\in \sub_2(I) \}$.
  \begin{sentences}\itemsep-2pt
  \item\label{maxK:struktura:1}
    If $L\in{\cal G}$ then $L \subset Z$.
  \item\label{maxK:struktura:2}
    Let $L\in\lines$. If $|L\cap Z|\geq 2$ then $L\in{\cal G}$.
  \item\label{maxK:struktura:3}
    $|Z_i| = n - m + 1$ for every $i\in I$.
  \item\label{maxK:struktura:4}
    Let $\{i,j\}\in\sub_2(I)$.
    Then $Z_i\cup \{q^{i,j} \}$ and $Z_j \cup \{ q^{i,j} \}$ are 
    two $K_{n-m+2}$-graphs  with the common sides through $q^{i,j}$. 
    Each of them is freely contained in $\goth M$.
  \item\label{maxK:struktura:5}
    $|Z| = \binom{n+1-m}{2}$
  \item\label{maxK:struktura:6}
    $|{\cal G}| = \binom{n+1-m}{3}$
  \item\label{maxK:struktura:7}
    Let $L\in{\cal G}_i$ for an $i\in I$. 
    Then $|L\cap X_i| = 2$ and $L\cap X_i \subset Z_i$, $|L\cap Z| =1$
  \item\label{maxK:struktura:8}
    Let $e\in\sub_2(Z_i)$ for an $i\in I$. Then $\overline{e}\in{\cal G}_i$.
  \item\label{maxK:struktura:9}
    $|{\cal G}_i| = \binom{|Z_i|}{2} = \binom{n+1-m}{2}$ for every $i\in I$.
  \item\label{maxK:struktura:10}
    Let $i\in I$. Through every point $p\in Z$ there passes exactly one $L\in{\cal G}_i$.
  \item\label{maxK:struktura:11}
    The structure $\struct{Z,{\cal G}}$ is a $\binkonf(n-m,+1)$-configuration
    regularly contained in $\goth M$.
  \end{sentences}
\end{prop}
\begin{proof}
\eqref{maxK:struktura:1}:\quad
  Suppose $a \in L\cap X_i$ for some $i\in I$ and $L\in{\cal G}$. 
  Comparing point-ranks we note that 
  all the lines of $\goth M$ through $a$ are the sides of $G_i$, so $L\in{\cal E}_i$:
  a contradiction.
\par\noindent
\eqref{maxK:struktura:2}:\quad
  Suppose $L\notin{\cal G}$, then there are $i\in I$ and an edge $e$ of $G_i$ such that 
  $L = \overline{e}$. Clearly, $|L\cap X_i|=2$, so $|L\cap Z|\leq 1$.
\par\noindent
\eqref{maxK:struktura:3}:\quad
  Evident: $Z_i = \{ x\in X_i\colon x \neq q^{i,k}, k\in I\setminus\{i\} \}$
  and $|I\setminus\{ i\}| = m-1$.
\par\noindent
\eqref{maxK:struktura:4}:\quad
  Evidently, any two points in $Z_i$ and any two points in $Z_j$ are on a line of $\goth M$:
  a suitable side of $G_i$ or of $G_j$ resp.
  Let $a\in Z_i$. Then 
  $\overline{q^{i,j},a}\setminus\{ q^{i,j},a \}$ is a point $b$ of $X_j$. 
  Suppose $b = q^{j,k}$  for some $k\in I$. 
  From \ref{prop:cross-compl-line}, $a=q^{i,k}$: a contradiction;
  thus $b\in Z_j$.
\par\noindent
\eqref{maxK:struktura:5}:\quad
  By \ref{prop:cross-compl2} and \ref{prop:cross-compl-line},  
  $|\bigcup_{i\in I}X_i| = m\cdot n - \binom{m}{2}\cdot 1 =: \gamma(n)$;
  we compute 
  $\binom{n+1}{2} - (\binom{m+1-n}{2} + \gamma(n))$ = 0.
\par\noindent
\eqref{maxK:struktura:6}:\quad
  Analogously, by \ref{prop:cross-compl-line} and \ref{cor:cross-compl-sides}, 
  $|\bigcup_{i\in I}{\cal E}_i| = 
  m\cdot\binom{n}{2} - \binom{m}{2}\cdot(n-1) + \binom{m}{3}\cdot 1 =:\delta(n)$,
  and then
  $\binom{n+1}{3} - (\binom{n+1-m}{3} + \delta(n)) = 0$.
\par\noindent
\eqref{maxK:struktura:7}:\quad
  It is clear that any $L\in {\cal G}_i\subset{\cal E}_i$ 
  crosses $X_i$ in a pair $a,b$ of points.
  Suppose $a\notin Z_i$. Then $a = q^{i,k}$ for some $k\in I$, $k\neq i$ and then 
  $L\in{\cal E}_i,{\cal E}_k$. This yields $a,b\in X_i$.
  Suppose $L = \{ a,b,c \}$ and $c \in X_k$ for $k\in I$. Then $L\in{\cal E}_k$, $k\neq i$,
  so $L\notin{\cal G}_i$.
\par\noindent
\eqref{maxK:struktura:8}:\quad
  Suppose $\overline{e}\notin{\cal G}_i$, so there is $k\neq i$ such that 
  $\overline{e}\in{\cal E}_k$. This means: $G_k$ contains an edge $e'$ with
  $\overline{e} = \overline{e'}$. 
  Then $e\cap e'\neq \emptyset$, so $e\cap X_k\neq \emptyset$,
  which contradicts $e\subset Z_i$.
\par\noindent
\eqref{maxK:struktura:9}:\quad
  Immediately follows form \eqref{maxK:struktura:7} and \eqref{maxK:struktura:8}.
\par\noindent
\eqref{maxK:struktura:10}:\quad
  In view of \eqref{maxK:struktura:7},
  the map ${\cal G}_i \ni L \longmapsto p\in L\cap Z$ is well defined.
  Clearly, it is injective. 
  From \eqref{maxK:struktura:9} and \eqref{maxK:struktura:3} it is also surjective, 
  and this is exactly the claim.
\par\noindent
\eqref{maxK:struktura:11}:\quad
  Immediate, after \eqref{maxK:struktura:1}, \eqref{maxK:struktura:2},
  \eqref{maxK:struktura:5}, and \eqref{maxK:struktura:6}.
\end{proof}

Let $I = \{ 1,\ldots,m \}$ be arbitrary, let $n > m$ be an integer, and let
$X$ be a set with $n-m+1$ elements.
Let us fix an arbitrary $\binkonf(n-m,+1)$-configuration 
${\goth B} = \struct{Z,{\cal G}}$.
Assume that we have two maps $\mu,\xi$ defined:
$\mu\colon I\longrightarrow Z^{\sub_2(X)}$ 
and
$\xi\colon I\times I\longrightarrow S_X$, such that
$\xi_{i,i} = \id$, $\xi_{i,j} = \xi_{j,i}^{-1}$, and
$\mu_i$ is a bijection for all $i,j\in I$.
Let $S = Z \cup (X\times I) \cup \sub_2(I)$ (to avoid silly errors we assume that
the given three sets are pairwise disjoint).
On $S$ we define the following family $\lines$ of blocks
\begin{eqnarray*}
  \lines & = & {\cal G} \\
  \strut & \cup & \text{ the lines of } \GrasSpace(I,2) \\
  \strut & \cup & \left\{ \{ \{i,j\}, (x,i), (\xi_{i,j}(x),j)  \}
                  \colon \{i,j\}\in\sub_2(I), x \in X  \right\} \\
  \strut & \cup & \left\{ \{ (a,i), (b,i), \mu_i(\{ a,b \}) \} 
                  \colon \{ a,b \}\in \sub_2(X), i\in I  \right\}.
\end{eqnarray*}
Write 
\begin{equation}\label{def:SPS}
 m \bowtie^\mu_\xi {\goth B} = \struct{S,\lines}.
\end{equation}
It needs only a straightforward (though quite tidy) verification to prove that

\begin{center}
{\em ${\goth M}:= m \bowtie^\mu_\xi {\goth B}$ is a $\binkonf(n,+1)$-configuration}. 
\end{center}

For each $i\in I$ we set $Z_i = X\times\{ i \}$, $S_i = \{ e\in\sub_2(I)\colon i\in e \}$,
and $X_i = Z_i \cup S_i$. It is seen that 
{\em $\goth M$ freely contains $m$ $K_n$-graphs $X_1,\ldots,X_m$}.
Indeed, let us write $a\oplus b = c$ when $\{a,b,c \}$ is a line (of the configuration in
question). Then we have $\{i,j\}\oplus\{i,k\} = \{j,k\}$,
$(a,i)\oplus(b,i) = \mu_i(\{a,b\})$, and $(a,i)\oplus\{i,j\} = (\xi_{i,j}(a),j)$.
It is seen that the point $\{i,j\}$ is the perspective center of two subgraphs $Z_i,Z_j$
of $\goth M$. So, we call the 
configuration $m \bowtie ^\mu_{\xi} {\goth B}$
{\em a system of perspective $(n-m+1)$-simplices}.
Define $\mu^o_i\colon\sub_2(Z_i)\longrightarrow Z$ by the formula
$\mu^o_i(\{ (x,i),(y,i) \}) = \mu(\{ x,y \})$; 
it is seen that the configuration $\goth B$ is the common `axis' of  the configurations 
$\struct{Z_i,\sub_2(Z_i)} +_{\mu^o_i} {\goth B}$
contained in $\goth M$.
\par
Note that the words `perspective', `axis', and `simplices' are used to suggest some 
{\em formal} similarities to objects considered in geometry.
Such a usage does not mean that the considered binomial configuration 
$m \bowtie^mu_\xi {\goth B}$ is necessarily realizable in a desarguesian projective
space.

Let us consider two special cases
of the above definition of a system of perspective simplices.
\begin{sentences}\itemsep-2pt
\item
  Let $m = n-1$. Then $\goth B$ consists of a single point $p$: the {\em center} of $\goth M$.
  Consequently, $\mu_i$ is constant, $\mu_i \equiv p$. Moreover, the set $X$ which 
  appears in the definition has 2 elements and then $|S_X| =2$. Then instead of a map $\xi$
  one can consider a graph ${\cal P}\subset\sub_2(I)$ defined by
  $\{ i,j\} \in {\cal P} \iff \xi_{i,j} = \id$. And then 
  {\em the system $m \bowtie^p_\xi \struct{\{ p \},\emptyset}$
  of $m$ perspective segments
  (of 2-simplices) is the multiveblen configuration 
   $\xwlepp({},{m},{\cal P},{},{\GrasSpace(m,2)})$}  
  (cf. \cite{pascvebl}, \cite{mveb2proj}).
\item
  Let $m = n-2$. Then $\goth B$ consists of a single $3$-point line, $L = \{a,b,c\}$.
  Up to a permutation of $X$ there is a unique bijection $\mu\colon\sub_2(X)\longrightarrow L$.
  Finally 
  {\em the system of $m$ perspective triangles 
  $m \bowtie^\mu_\xi \struct{L,\{ L \}}$
  is the system of triangle perspectives
    $\xwlepq({I},{},{\xi},{},{\GrasSpace(I,2)})$}
  (cf. \cite{linstp}).
\end{sentences}

Now, till the end of this section writing
`a configuration contains $m$ graphs' we mean
`a configuration contains {\em at least} $m$ graphs'.
\begin{thm}\label{thm:SPS}
  Let ${\goth M}$ be a $\binkonf(n,+1)$-configuration.
  The following conditions are equivalent.
  \begin{enumerate}[(i)]\itemsep-2pt
  \item
    $\goth M$ freely contains $m$ $K_n$-graphs.
  \item
    $\goth M$ is a system of $m$ perspective $(n-m+1)$-simplices i.e. 
    ${\goth M} \cong m \bowtie^\mu_\xi {\goth B}$ for a 
    $\binkonf(n-m,+1)$-configuration $\goth B$ and some (admissible) maps $\mu,\xi$.
  \end{enumerate}
\end{thm}
\begin{proof}
  We have already noticed that $m \bowtie^\mu_\xi {\goth B}$ contains required subgraphs.
  \par
  Let ${\goth M} = \struct{S,\lines}$ and
  let $X_1,\ldots,X_m\in\sub_n(S)$ be pairwise distinct. Assume that 
  $G_i = \struct{X_i,\sub_2(X_i)}$ is freely contained in $\goth M$ for every 
  $i = 1,\ldots,m$.
  Let us adopt the notation of \ref{lem:maxK:struktura}.
  
  Set ${\goth B} = \struct{Z,{\cal G}}$.
  Let us fix a $(n-m+1)$-element set $X$ and let 
  $\nu_i\colon Z_i \longrightarrow X$ be a fixed bijection for each $i\in I$.
  Let $a,b, \in X$ and $i,j\in I$.
  Define 
  \begin{ctext}
    if $a\neq b$ then 
    $\mu_i(\{ a,b \}) = \overline{\nu_i(a)\nu_i(b)}\setminus \{ \nu_i(a),\nu_i(b) \}$,
  \end{ctext}
  $x_{i,i} = \id_X$, and 
  \begin{ctext}
    if $i\neq j$ then
    $\xi_{i,j}(a) = b$ iff $\{q^{i,j},\nu_i(a),\nu_j(b)\}\in\lines$.
  \end{ctext}
  Finally, we define on the points of $\goth M$ the following map $F$:
  $$
    F \colon \left\{
    \begin{array}{rcl}
      Q\ni q^{i,j} & \longmapsto & \{i,j\} \\
      Z_i \ni x & \longmapsto & (x,i)  \\
      Z \ni a & \longmapsto & a
    \end{array} 
    \right.
  $$
  It is a standard student's exercise to compute that $F$ is an isomorphism of $\goth M$ and 
  $m \bowtie^\mu_\xi {\goth B}$.
\end{proof}
\begin{cor}\label{cor:maxK:pod}
  Let $\goth M$ be a $\binkonf(n,+1)$-configuration.
  \begin{enumerate}[(i)]\itemsep-2pt
  \item
    $\goth M$ freely contains $n-1$ graphs $K_n$ iff $\goth M$ is (isomorphic to)
    a multiveblen configuration.
  \item
    $\goth M$ freely contains $n-2$ graphs $K_n$ iff $\goth M$ is (isomorphic to)
    a system of triangle perspectives.
  \end{enumerate}
\end{cor}

Particular instances of \ref{cor:maxKgras} and \ref{cor:maxK:pod} in case $n=4$ can be 
found in \cite{klik:VC}: 
a $10_3$ configuration contains four $K_4$ iff it contains five $K_4$
iff it is a Desargues Configuration; 
a $10_3$ configuration contains three $K_4$ iff it is a 
multiveblen configuration i.e. iff 
it is the Desargues or it is the Kantor $10_3 G$-configuration (cf. \cite{kantor});
a $10_3$ configuration contains two $K_4$ iff it is a system of triangle perspectives
i.e. 
it is one of the following: the Desargues, the Kantor $10_3G$, or the fez configuration.

\section{Existence problems}

\begin{prop}\label{jedenup}
  If there is a $\binkonf(n,0)$-configuration which freely contains exactly
  $m$ maximal complete subgraphs where $m\leq n-2$ then there exists
  a $\binkonf(n,1)$-configuration which freely contains exactly $m+1$ 
  maximal complete subgraphs.
\end{prop}
\begin{proof}
  Let $Y_1,\ldots,Y_m$ be the $K_{n-1}$-subgraphs of a 
  $\binkonf(n,0)$-configuration ${\goth M} = \struct{S,\lines}$.
  Set $I = \{ 1,\ldots,m \}$.
  Let us reprezent $\goth M$ as a system of perspectives, so
  let $q_{i,j} \in Y_i \cap Y_j$ for distinct $i,j$ and 
  $Q= \{ q_{i,j}\colon \{i,j\}\in \sub_2(I) \}$, 
  $G_i = Y_i\setminus Q$, and let $Z$ be an ``axis" i.e. the intersection
  of all the complementary subconfigurations of the $Y_i$'es.
  Let $X$ be an arbitrary set disjoint with $S$ of cardinality $n$
  and let $P\in\sub_m(X)$. Let us number the elements of $P$:
  $P = \{ p_1,\ldots,p_m \}$ and the elements of $X\setminus P$:
  $X\setminus P = \{ x_1,\ldots,x_{n-m} \}$. Note that $n-m\geq 2$.
  For each $\{i,j\} \in \sub_2(I)$ we introduce the triple 
  $\{ p_i,p_j,q_{i,j} \}$ as a new line. 
  The number of points in each of the sets $G_i$ is $n-m$; let 
  $\sigma_i$ be an arbitrary bijection of $G_i$ onto $X\setminus P$.
  The second family of new lines consists of the triples
  $\{ p_i,x,\sigma_i(x) \}$ with $x\in G_i$, $i\in I$.
  Finally, the third class of the new lines consists of the triples
  $\{ x_i,x_j,\mu(x_i,x_j) \}$, where $\mu$ is an arbtrary labelling of
  the edges of the graph $K_{X\setminus P}$ by the elements of $Z$:
  it is possible due to cardinalities of the sets in question.
  Let ${\goth M}^*$ be the structure defined on the point universe
  $S\cup X$, whose lines are the lines of $\goth M$ and the three classes
  of new lines introduced above. It is seen that ${\goth M}^*$ is
  a $\binkonf(n,1)$-configuration. It is also evident that $K_X$ and
  $K_{Y_i\cup\{ p_i \}}$ for $i\in I$  are $K_n$-subgraphs freely contained in
  ${\goth M}^*$.
  Suppose that ${\goth M}^*$ contains another freely contained $K_n$-subgraph $K_Y$,
  let $p\in X\cap Y$. Then $Y_0:= Y\setminus X = Y\setminus \{ p \}$ is a 
  $K_{n-1}$ subgraph of $\goth M$. Consequently, $Y_0 = Y_i$ for some $i\in I$.
  Suppose that $p\neq p_i$, then the two subgraphs $K_{Y_i\cup\{ p_i \}}$
  and $K_Y$ freely contained in ${\goth M}^*$ have more than a point in common.
  Consequently, $p=p_i$ and $Y = Y_i\cup\{ p_i \}$.
  Thus ${\goth M}^*$ freely contains exactly $m+1$ $K_n$-subgraphs.
\end{proof}

\begin{prop}\label{dwadown}
  If there exists a $\binkonf(n,+1)$-configuration which freely contains exactly two
  $K_n$ graphs then there is also a $\binkonf(n,1)$-configuration
  without any $K_n$-subgraph freely contained in it.
\end{prop}
\begin{proof}
  \def\Mx{{\goth M}^\ast}
  Let $n \geq 4$. Let ${\goth M} = \struct{S,\lines}$ be a $\binkonf(n,+1)$-configuration
  which freely contains exactly two complete $K_n$-graphs $X_1,X_2$.
  Let $p\in X_1\cap X_2$. Set $A_i = X_i\setminus\{ p \}$.
  Let $e_1\in\sub_2(A_1)$ and $e_2\in\sub_2(A_2)$ such that 
  $\overline{e_1}\cap\overline{e_2}\ni q$ for a point $q$ 
  (cf. \ref{prop:cross-compl3}\eqref{cros-compl3:2})
  Let $e_1 = \{a_1,b_1\}$, $e_2 = \{a_2,b_2\}$ such that 
  $b_2\notin\overline{p,a_1}$ and $a_2\notin\overline{p,b_1}$.
  We replace the two lines
  $\{ q,a_1,b_1 \}$ and $\{ q,a_2,b_2 \}$ of $\goth M$
  by two other triples
  $\{ q,a_1,b_2 \}$ and $\{ q,b_1,a_2 \}$;
  let $\Mx$ be the obtained incidence structure. 
  Clearly, $\Mx$ is a $\binkonf(n,+1)$-configuration.
  {\em $\Mx$ does not freely contain any $K_n$-graph}
  \par\noindent{\sc Indeed}:
  Suppose that $\Mx$ freely contains a $K_n$-graph $Y$.
  Let us have a look at the collinearity graph $A_{\goth K}$ of an arbitrary configuration
  $\goth K$.
  Clearly, if $K_X$ is contained in $\goth K$ then $K_X$ is a subgraph of $A_{\goth K}$.
  In our case exactly two edges $e_1,e_2$ of $A_{\goth M}$ were replaced by two (other) edges
  $\{ a_1,b_2 \}$, $\{ a_2,b_1 \}$ to form $A_{\Mx}$.
  So, $K_Y$ cannot be build entirely from the edges missing $e_1\cup e_2$.
  \par
  (1) $Y\cap e_1 \neq \emptyset \neq Y\cap e_2$. Indeed, suppose, eg. $Y\cap e_1 = \emptyset$.
  The same pairs of points in $S\setminus e_1$ (except $e_2$) are collinear in 
  $\goth M$ and in $\Mx$ and therefore $Y$ is a $K_n$-graph in $\goth M$: a contradiction.
  \par
  (2) $e_1\not\subset Y$ and $e_2\not\subset Y$: the pair of points in $e_1$ is not
  collinear in $\Mx$, and, analogously the points in $e_2$ are also not collinear.
  \par
  So, $Y$ contains exactly one point $x_1$ in $e_1$ and one point $y_2$ in $e_2$.
  Clearly, $x_1,y_2$ must be collinear in $\Mx$.
  \par
  (3) Suppose that $y_2 \in \overline{p,x_1}$. Without loss of generality we can 
  take $x_1 = a_1$, $y_2 = a_2$. Then $b_1,b_2\notin Y$.
  For points in $S\setminus \{ b_1,b_2 \}$ exactly the same pairs are collinear in 
  $\goth M$ and in $\Mx$, so $Y$ is a subgraph of $\goth M$, which is impossible.
  \par
  Without loss of generality we can assume that $a_1,b_2 \in Y$ and $a_2,b_1\notin Y$.
  The following three cases should be considered:
  \begin{enumerate}[(1)]\setcounter{enumi}{3}\itemsep-2pt
  \item\label{bzz1}
    $a_2 \in \overline{p,a_1}$, and $b_2 \in \overline{p,b_1}$,
  \item\label{bzz2}
    $a_2 \in \overline{p,a_1}$ and $b_2 \notin \overline{p,b_1}$,
  \item\label{bzz3}
    $a_2 \notin \overline{p,a_1}$ and $b_2 \notin \overline{p,b_1}$.
  \end{enumerate}
  \par\eqref{bzz1}:
    Note that the sides of $Y$ and the lines of $\Mx$ through vertices of $Y$ coincide. 
    So, $\{ p, a_1\}$ is an edge of $Y$ and thus $p\in Y$.
    Take any point $c_1\in A_1\setminus e_1$.
    Then $c_1,b_2$ are not collinear in $\Mx$, so $c_1\notin Y$. 
    Let $c_2\in \overline{p,c_1}\setminus\{ p,c_1 \}$ (this line of $\goth M$ 
    was unchanged), then $c_2\in Y$.
    But $c_2$ and $a_1$ are not collinear in $\Mx$ and a  contradiction arizes.
  \par\eqref{bzz2}:
    In this case also necessarily $p\in Y$. 
    Let $c_2\in\overline{p,b_1}\setminus\{ p,b_1 \}$; then $c_2\neq b_2$ and $c_2\in Y$.
    But, contradictory, $a_1,c_2$ are not collinear in $\Mx$.
  \par\eqref{bzz3}:
    Now, either $p\in Y$ or $c_2 \in Y$,
    where $\{p,a_1,c_2\}\in\lines$ ($c_2\in X_2$). 
    If $p\in Y$ then we take 
    $c_1\in \overline{p,a_2}\setminus\{ p,a_2 \}$;
    then $c_1\in Y$. An inconsistency appears, as $c_1,b_2$ are not collinear in $\Mx$.
    Consequently, $c_2\in Y$ and $p\notin Y$.
    Therefore, the point $c_1$ in $\overline{p,b_2}\setminus\{p,b_2\}$ is in $Y$ 
    ($c_1 \in X_1$). But $c_1,c_2$ are not collinear in $\Mx$, though.
\end{proof}

As an important consequence we obtain now
\begin{thm}\label{thm:allKn}
  Let $m,n$ be integers, $4\leq n$, and $1\leq m \leq n-1$ or $m=n+1$.
  Then there exists a $\binkonf(n,+1)$-configuration which freely contains
  exactly $m$ $K_n$-graphs.
\end{thm}
\begin{proof}
 Let $n\geq 4$ be an integer and let 
 $J(n)$ be the set of integers $m$ such that there is a $\binkonf(n,1)$- configuration
 with  exactly $m$ freely contained subgraphs $K_n$.
 From \ref{prop:maxK}, $J(n)\subseteq \{ 0,1,\ldots,n-1,n+1 \} =: F(n)$.
 We need to prove that $J(n) = F(n)$ for every integer $n\geq 4$.
 Clearly, this equality holds for $n=4$ (cf. \cite{klik:VC}).
 \par
 Assume that $J(n) = F(n)$ holds for an integer $n$. 
 From \ref{jedenup} and  \ref{cor:maxKgras}, 
 $J(n+1)  \supseteq \{ 1,2,\ldots,n,n+2 \}$.
 Then from \ref{dwadown} we get $0\in J(n+1)$ and therefore
 $J(n+1) = F(n+1)$. By induction, we are done.
\end{proof}
\begin{rem}
  There is no reasonable $\binkonf({0},2)$-configuration,
  there is exactly one $\binkonf({0},3)$-configuration: a line
  $\GrasSpace(3,2)$, with exactly $3$ freely contained copies of $K_2$,
  and there is exactly one $\binkonf({0},4)$-configuration:
  the Veblen configuration $\GrasSpace(4,2)$, 
  which freely contains 4 copies of $K_3$. 
\end{rem}

\section{Other known examples}

\subsection{Combinatorial Veronesians}

Let us adopt the notation of \cite{combver}.
Let $|X| = 3$.
Then the combinatorial Veronesian $\VerSpace(X,k)$ is 
a $\binkonf(k,+2)$-configuration; 
its point set is the set $\msub_k(X)$ of the $k$-element multisets with 
elements in $X$.
The maximal cliques of $\VerSpace(X,k)$ were established in 
\cite{veradjac}. From that results we learn that
\begin{fact}
  The $K_{k+1}$ graphs freely contained in $\VerSpace(X,k)$ are the sets
  $X_{a,b} := \msub_k(\{ a,b \})$, 
  $X_{b,c} := \msub_k(\{ b,c \})$, and 
  $X_{c,a} := \msub_k(\{ c,a \})$.
  Its axis is the set $X^k$. For $u\in\sub_2(X)$ the $\binkonf(k,+1)$-subconfiguration
  of $\VerSpace(X,k)$ complementary to $X_u$ 
  (with the universe $y\msub_{k-1}(X)$, $y\in X\setminus u$) 
  is isomorphic to $\VerSpace(X,k-1)$.
\end{fact}
\begin{cor}
  A $\binkonf(k,+1)$-Veronesian with $k>2$ 
  contains exactly three complete
  $K_{k}$-graphs freely contained in it.
\end{cor}

\subsection{Quasi Grassmannians}

Let us adopt the notation of \cite{skewgras}.
The configuration ${\goth R}_n$ is a $\binkonf(n,+2)$-configuration.
Recall the role of the set $X = \{1,2 \}$ if $n$ is even and 
$X= \{ 0,1,2 \}$ if $n$ is odd. 
Namely, let us cite after \cite{skewgras} the following
\begin{fact}
  The complete $K_{n+1}$-graphs freely contained in ${\goth R}_n$ are the sets
  $S(i) = \{a\in\sub_2(Y)\colon i\in a\}$, where $\sub_2(Y)$ is the point set
  of ${\goth R}_n$, $X\subset Y$, and $i \in X$.
\end{fact}
\begin{cor}
  Let $\goth M$ be a $\binkonf(n,0)$-quasi-Grassmannian, $n>2$. 
  If $n$ is even then $\goth M$
  freely contains exactly two $K_{n-1}$-graphs, and it freely contains exactly three
  $K_{n-1}$-graphs when $n$ is odd.
\end{cor}
Besides, the above indicates one more similarity between 
combinatorial Veronesians and quasi Grassmannians represented
as a fan of configurations $10_3G$.  



\bigskip
\begin{small}
\noindent
Authors' address:\\
Ma{\l}gorzata Pra{\.z}mowska, Krzysztof Pra{\.z}mowski\\
Institute of Mathematics, University of Bia{\l}ystok\\
ul. Akademicka 2, 15-246 Bia{\l}ystok, Poland\\
{\ttfamily malgpraz@math.uwb.edu.pl},
{\ttfamily krzypraz@math.uwb.edu.pl}
\end{small}

\end{document}